\documentclass[a4paper,12pt]{amsart}

\usepackage{amsfonts}
\usepackage{amsmath}
\usepackage{amssymb}

\usepackage{mathrsfs}
\usepackage[colorlinks]{hyperref}

\numberwithin{equation}{section}
\theoremstyle{plain}
\newtheorem{thm}{Theorem}[section]
\newtheorem{prop}[thm]{Proposition}
\newtheorem{cor}[thm]{Corollary}

\theoremstyle{definition}

\newcommand{\Dcal}{\mathcal D}
\def\H{{\mathcal L}}
\def\R{\mathcal R}
\def\Sp{{\mathbb H}^{n}}

\def\e[#1]{{\textrm{e}}^{#1}}

\def\Rn{{\mathbb R}^n}
\def\G{{\mathbb G}}
\def\hil{{\mathtt{HS}}}

\begin{document}

\title[Nonlinear wave equations on the Heisenberg group]
{
Nonlinear damped wave equations for the sub-Laplacian on the Heisenberg group and for Rockland operators on graded Lie groups}

\author[Michael Ruzhansky]{Michael Ruzhansky}
\address{
  Michael Ruzhansky:
  \endgraf
  Department of Mathematics
  \endgraf
  Imperial College London
  \endgraf
  180 Queen's Gate, London, SW7 2AZ
  \endgraf
  United Kingdom
  \endgraf
  {\it E-mail address} {\rm m.ruzhansky@imperial.ac.uk}
  }
\author[Niyaz Tokmagambetov]{Niyaz Tokmagambetov}
\address{
  Niyaz Tokmagambetov:
  \endgraf
    al--Farabi Kazakh National University
  \endgraf
  71 al--Farabi ave., Almaty, 050040
  \endgraf
  Kazakhstan,
  \endgraf
   and
  \endgraf
    Department of Mathematics
  \endgraf
  Imperial College London
  \endgraf
  180 Queen's Gate, London, SW7 2AZ
  \endgraf
  United Kingdom
  \endgraf
  {\it E-mail address} {\rm n.tokmagambetov@imperial.ac.uk}
 }

\thanks{The authors were supported in parts by the EPSRC
grant EP/K039407/1 and by the Leverhulme Grant RPG-2014-02,
as well as by the MESRK (Ministry of Education and Science of the Republic of Kazakhstan) grant 0773/GF4. No new data was collected or generated during the course of research.}

\date{\today}


\begin{abstract}
In this paper we study the Cauchy problem for the semilinear damped wave equation for the sub-Laplacian on the Heisenberg group. In the case of the positive mass, we show the global in time well-posedness for small data for power like nonlinearities. We also obtain similar well-posedness results for the wave equations for Rockland operators on general graded Lie groups. In particular, this includes higher order operators on $\mathbb R^n$ and on the Heisenberg group, such as powers of the Laplacian or of the sub-Laplacian.
In addition, we establish a new family of Gagliardo-Nirenberg inequalities on а graded Lie groups that play a crucial role in the proof, but which are also of interest on their own:
if $\G$ is a graded Lie group of homogeneous dimension $Q$ and
$a>0$, $
1<r<\frac{Q}{a},$ 
and $1\leq p\leq q\leq \frac{rQ}{Q-ar},
$
then we have the following Gagliardo-Nirenberg type inequality
$$
\|u\|_{L^{q}(\G)}\lesssim \|u\|_{\dot{L}_{a}^{r}(\G)}^{s} \|u\|_{L^{p}(\G)}^{1-s}
$$
for $s=\left(\frac1p-\frac1q\right) \left(\frac{a}Q+\frac1p-\frac1r\right)^{-1}\in [0,1]$
provided that $\frac{a}Q+\frac1p-\frac1r\not=0$, where $\dot{L}_{a}^{r}$ is the homogeneous Sobolev space of order $a$ over $L^r$.
If $\frac{a}Q+\frac1p-\frac1r=0$, we have $p=q=\frac{rQ}{Q-ar}$, and then the above inequality holds for any $0\leq s\leq 1$.
\end{abstract}

\maketitle

\tableofcontents

\section{Introduction}

In this paper we investigate the global in time well-posedness for the damped wave equation for the sub-Laplacian on the Heisenberg group. Strichartz estimates for the wave equation for the sub-Laplacian on the Heisenberg group have been analysed by Bahouri, G\'erard and Xu in \cite{BGX-Heis} where a weak decay rate in dispersive estimates was established.
Recently, such results were extended to step 2 stratified Lie groups by
Bahouri, Fermanian-Kammerer and Gallagher \cite{BFG16} where it was shown that the decay rate of solution may depend on the dimension of the centre of the group.
Wave equations for the full Laplacian on the Heisenberg group have been investigated in \cite{FMV-Hei-full}, \cite{Liu-Song:full} where better decay rates have been obtained.

One purpose of this paper is to investigate the global in time well-posedness of the Cauchy problem for the semilinear damped wave equation
\begin{equation}\label{EQ: NoL-01i}
\left\{ \begin{split}
\partial_{t}^{2}u(t)-\H u(t)+b\partial_{t}u(t)+m u(t)&=f(u), \quad t>0,\\
u(0)&=u_{0}\in L^{2}(\Sp), \\
\partial_{t}u(0)&=u_{1}\in L^{2}(\Sp),
\end{split}
\right.
\end{equation}
with the damping term determined by $b>0$ and with the mass $m>0$, where $\Sp$ is the Heisenberg group and $\H$ is the sub-Laplacian. Consequently, we also establish similar results for a more general setting, namely, when the Heisenberg group $\Sp$ is replaced by a general graded Lie group $\G$, and the sub-Laplacian $\H$ is replaced by an arbitrary Rockland operator $\R$, i.e. by an arbitrary left-invariant homogeneous hypoelliptic differential operator.

The nonlinearity $f$ in this paper will be assumed to satisfy, for some $p>1$, the conditions
\begin{equation} \label{PR: f-007i}
\left\{
\begin{split}
f(0) & ={0}, \\
|f(u)-f(v)| & \leq C (|u|^{p-1}+|v|^{p-1})|u-v|.
\end{split}
\right.
\end{equation}
In particular, this includes the cases
\begin{equation}\label{EQ:nonlinexi}
f(u)=\mu |u|^{p-1} u, \quad \textrm{ for some } p>1,\quad \mu\in\mathbb C,
\end{equation}
as well as the more general case of differentiable functions $f$ satisfying
\begin{equation}\label{EQ:nonlinex2i}
|f'(u)|\leq C|u|^{p-1}.
\end{equation}

To fix the notation concerning the equation \eqref{EQ: NoL-01i}, for $n\in\mathbb N$, the Heisenberg group $\Sp$ is the manifold $\mathbb R^{2n+1}$ endowed with the group structure
$$
(x, y, t)\circ(x', y', t'):=(x+x', y+y',t+t'+\frac{1}{2}(x\cdot y'-x'\cdot y)),
$$
where $(x,y,t)$ and $(x', y', t')$ are in $\mathbb R^{n}\times\mathbb R^{n}\times\mathbb R\sim\Sp.$ The sub--Laplacian on the Heisenberg group $\Sp$ is given by
\begin{equation}\label{EQ: sub--Laplacian}
\H:=\sum_{j=1}^{n}(X_{j}^{2}+Y_{j}^{2}), \quad\textrm{ with }
X_j:=\partial_{x_{j}}-\frac{y_{j}}{2}\partial_{t},\;\; Y_j:=\partial_{y_{j}}+\frac{x_{j}}{2}\partial_{t}.
\end{equation}
In this case, in Theorem \ref{TH: 01} we will show the global in time well-posedness of the Cauchy problem
\eqref{EQ: NoL-01i}:
\begin{itemize}
\item for small data
$(u_{0},u_1)\in H_{\H}^{1}(\Sp)\times L^2(\Sp)$,
\item and for nonlinearities $f(u)$ satisfying \eqref{PR: f-007i} for $1 < p \leq 1 + 1/n$.
\end{itemize}
Here $H_{\H}^{1}(\Sp)$ denotes the sub-Laplacian Sobolev space, analysed by Folland \cite{F75}.
Consequently, we extend this result beyond the setting of the Heisenberg group and second order operators, in a way that we now describe.

Following Folland and Stein \cite{FS-book}, we recall that $\G$ is a graded Lie group if there is a gradation on its Lie algebra $\mathfrak g$, i.e. a vector space decomposition
$$\mathfrak g=\bigoplus_{j=1}^\infty V_j\;\textrm{ such that } [V_i,V_j]\subset V_{i+j},$$
will all but finitely many of $V_j$ being zero,
see Section \ref{SEC:prelim} for a precise definition. This leads to a family of dilations on it with rational weights, compatible with the group structure. If $V_1$ generates $\mathfrak g$ as an algebra, the group is said to be stratified and the sum of squares of a basis of vector fields in $V_1$ yields a sub-Laplacian on $\G$. However, non-stratified graded $\G$
may not have a homogeneous sub-Laplacian or Laplacian but they always have so-called Rockland operators. Such operators appeared in the hypoellipticity considerations by Rockland \cite{Rockland} defined by the condition that their infinitesimal representations are injective on smooth vectors. Suitable partial reformulations of this conditions were further proposed by Rockland \cite{Rockland} and Beals \cite{Beals-Rockland}, until the resolution in \cite{HN-79} by Helffer and Nourrigat of what has become known as the Rockland conjecture, and what we can adopt as the definition here:

{\em Rockland operators are left-invariant homogeneous hypoelliptic differential operators on $\G$.}

In fact, the existence of such operators on nilpotent Lie groups singles out the class of graded groups, see \cite{Miller:80}, \cite{tER:97}. In the realm of homogeneous Lie groups, the graded groups can be also characterised by dilations having rational weights, see \cite[Section 4.1.1]{FR16}.

Thus, in our extension of the obtained result from the Heisenberg group to graded Lie groups, we will work with positive Rockland operators $\R$. To give some examples, this setting includes:

\begin{itemize}
\item for $\G=\Rn$, $\R$ may be any positive homogeneous elliptic differential operator with constant coefficients. For example, we can take
$$\R=(-\Delta)^m \;\textrm{ or }\; \R=(-1)^m\sum_{j=1}^n a_j\left(\frac{\partial}{\partial x_j}\right)^{2m},\quad a_j>0,\; m\in\mathbb N;$$
\item for $\G=\Sp$ the Heisenberg group, we can take
$$\R=(-\H)^m \;\textrm{ or }\; \R=(-1)^m\sum_{j=1}^n (a_j X_j^{2m}+b_j Y_j^{2m}),\quad
a_j,b_j>0,\; m\in\mathbb N,$$
where $\H$ is the sub-Laplacian and $X_j, Y_j$ are the left-invariant vector fields in \eqref{EQ: sub--Laplacian}.
\item for any stratified Lie group (or homogeneous Carnot group) with vectors $X_1,\ldots,X_k$ spanning the first stratum, we can take
$$
\R=(-1)^m\sum_{j=1}^k a_j  X_j^{2m},\quad
a_j>0,\;
$$
so that, in particular, for $m=1$, $\R$ is a positive sub-Laplacian;
\item for any graded Lie group $\G\sim\Rn$ with dilation weights $\nu_1,\ldots,\nu_n$ let us fix the basis $X_1,\ldots,X_n$ of the Lie algebra $\mathfrak g$ of $\G$ satisfying
$$
D_r X_j=r^{\nu_j} X_j,\quad j=1,\ldots,n,\; r>0,
$$
where $D_r$ denote the dilations on the Lie algebra. If $\nu_0$ is any common multiple of $\nu_1,\ldots,\nu_n$, the operator
$$
\R=\sum_{j=1}^n (-1)^{\frac{\nu_0}{\nu_j}} a_j X_j^{2\frac{\nu_0}{\nu_j}},\quad a_j>0,
$$
is a Rockland operator of homogeneous degree $2\nu_0$. The Rockland operator can be also adapted to a special selection of vector fields generating the Lie algebra in a suitable way, such as the vector fields from the first stratum on the stratified Lie groups. We refer to \cite[Section 4.1.2]{FR16} for many other examples and a discussion of Rockland operators.
\end{itemize}


In the setting of a general graded Lie group $\G$ of homogeneous dimension $Q$, 
which is defined by
$$
Q=\nu_1+\ldots+\nu_n,
$$
we consider the nonlinear damped wave equation for a positive Rockland operator $\R$ of homogeneous degree $\nu$,
\begin{equation}\label{EQ: NoL-01iR}
\left\{ \begin{split}
\partial_{t}^{2}u(t)+\R u(t)+b\partial_{t}u(t)+m u(t)&=F(u, \{\R^{j/\nu}u\}_{j=1}^{\left[\frac{\nu}{2}\right]-1}), \quad t>0,\\
u(0)&=u_{0}\in L^{2}(\G), \\
\partial_{t}u(0)&=u_{1}\in L^{2}(\G),
\end{split}
\right.
\end{equation}
with the damping term determined by $b>0$ and with the mass $m>0$. Here $[\frac{\nu}{2}]$ stands for the integer part of $\frac{\nu}{2}.$
In this case, in Theorem \ref{TH: 01-g} and Theorem \ref{TH: 01-G} we will show the global in time well-posedness of the Cauchy problem
\eqref{EQ: NoL-01iR}:
\begin{itemize}
\item for small data
$(u_{0},u_1)\in H^{\nu/2}(\G)\times L^2(\G)$,
\item and for nonlinearities
$F:\mathbb C^{\left[\nu/2\right]}\to\mathbb C$ with the following property:
\begin{equation} \label{RM: F-g-i}
\left\{
\begin{split}
F(0) & =0, \\
|F(U)-F(V)| & \leq C (|U|^{p-1}+|V|^{p-1})|U-V|,
\end{split}
\right.
\end{equation}
where $U=(\{\R^{j/\nu}u\}_{j=0}^{\left[\frac{\nu}{2}\right]-1})$,
for $1 < p \leq 1+\frac{2}{Q-2}$.
\end{itemize}
Here $H^{\nu/2}(\Sp)$ denotes the Sobolev space of order $\nu/2$ associated to $\R$, analysed in \cite{FR:Sobolev} and in \cite[Section 4.4]{FR16}.

In the case of the Heisenberg group $\G=\Sp$ and $\R=-\H$, we have $\nu=2$ and $Q=2n+2$, and this result recaptures the first result in Theorem \ref{TH: 01} in this setting. Moreover, on stratified groups, i.e. with $\nu=2$, this gives the class of semilinear equations in Theorem \ref{TH: 01-g}.

However, to simplify the exposition, we give a detailed proof in the case of the sub-Laplacian on the Heisenberg group, and then indicate the necessary modifications for the case of general positive Rockland operators on general graded Lie groups.

In both cases of $\Sp$ and more general graded Lie groups $\G$, our proof relies on the group Fourier analysis on $\G$ to obtain the exponential time decay for solutions to the linear problem. This is possible due to the inclusion of positive mass term $m>0$ leading to the separation of the spectrum of $\R$ and of its infinitesimal representations from zero. Consequently, the nonlinear analysis relies on the application of the Gagliardo-Nirenberg inequality on $\G$. While such inequality is well-known on the Heisenberg group $\Sp$, the known graded group versions in \cite{BFKG-graded} or on \cite{FR16} are not suitable for our analysis. Thus, in Theorem \ref{LM: GN-Ineq-s on graded Lie Groups} we derive the necessary version of the Gagliardo-Nirenberg inequality based on the graded group version of Sobolev inequality established in \cite{FR16}.
More generally, we show that if $\G$ is a graded Lie group of homogeneous dimension $Q$ and
$$
a>0,\quad 1<r<\frac{Q}{a}\; \textrm{ and }\; 1\leq p\leq q\leq \frac{rQ}{Q-ar},
$$
then we have the following Gagliardo-Nirenberg type inequality
\begin{equation}
\label{EQ: GN-inequality on graded Lie Groups-i}
\|u\|_{L^{q}(\G)}\lesssim \|u\|_{\dot{L}_{a}^{r}(\G)}^{s} \|u\|_{L^{p}(\G)}^{1-s}\simeq
 \|\R^{a/\nu}u\|_{L^{r}(\G)}^{s} \|u\|_{L^{p}(\G)}^{1-s},
\end{equation}
for $s=\left(\frac1p-\frac1q\right) \left(\frac{a}Q+\frac1p-\frac1r\right)^{-1}\in [0,1]$
provided that $\frac{a}Q+\frac1p-\frac1r\not=0$.
Here $\dot{L}_{a}^{r}(\G)$ is the homogeneous Sobolev space of order $a$ over $L^r(\G)$, and we refer to \cite{FR:Sobolev} and \cite[Section 4.4]{FR16} for an extensive analysis of these spaces and their properties in the setting of general graded Lie groups.

If $\frac{a}Q+\frac1p-\frac1r=0$, we have $p=q=\frac{rQ}{Q-ar}$, and then \eqref{EQ: GN-inequality on graded Lie Groups-i} holds for any $0\leq s\leq 1$.

The Fourier analysis we use follows the pseudo-differential analysis
as described, for example, in \cite{Folland:bk}, \cite{BFKG-Ast} or \cite{FR:CRAS}
in the case of the Heisenberg group, and in \cite{FR16} on general graded Lie groups, see also \cite{FR:graded}.

The similar strategy of obtaining $L^2$-estimates for solutions of linear problems has been used in \cite{GR15} in the analysis of weakly hyperbolic wave equations for the sub-Laplacians on compact Lie groups.
Some techniques of similar type also appear in the analysis of general operators with discrete spectrum and with time-dependent coefficients in \cite{RT-IMRN} and \cite{RT-LMP}.
Estimates in $L^p$ for solution of the wave equation for the sub-Laplacian on the Heisenberg group
were considered in \cite{Muller-Stein:Lp-wave-Heis}, and on groups of Heisenberg type in \cite{MS-wave-Lp}.
The potential theory and functional estimates in the setting of stratified groups have been recently analysed in \cite{RS-AM,RS-JDE,RS-PAMS}.

Throughout this paper we will often use the notation $\lesssim$ instead of $\leq$ to avoid repeating the constants which are not dependent on the main parameters, especially, on functions appearing in the estimates.

\section{Linear damped wave equation on the Heisenberg group}
\label{SEC:linear}

In what follows, we will need some elements of the analysis on the Heisenberg group $\Sp$. It will be convenient for us to follow the notations from \cite[Chapter 6]{FR16} to which we refer for further details. We start by recalling the definition of the group Fourier transform on $\Sp$. For $f\in\mathcal S(\Sp)$ we denote its group Fourier transform by
\begin{equation}\label{EQ: Group-FT}
\widehat{f}(\lambda):=\int_{\Sp}f(x)\pi_{\lambda}(x)^{\ast}dx
\end{equation}
with the Schr\"{o}dinger representations
\begin{equation}\label{EQ: Schrodinger representation}
\pi_{\lambda}: L^{2}(\mathbb R^{n}) \to L^{2}(\mathbb R^{n})
\end{equation}
for all $\lambda\in\mathbb R^{\ast}:=\mathbb R\setminus\{0\}$.
The Fourier inversion formula then takes the form
\begin{equation}\label{EQ: Group-FT}
f(x)=\int_{\lambda\in\mathbb R^{\ast}}\mathrm{Tr}[\widehat{f}(\lambda)\pi_{\lambda}(x)]
\, |\lambda|^{n}d\lambda.
\end{equation}
where $\mathrm{Tr}$ is the trace operator.
The Plancherel formula becomes
\begin{equation}\label{EQ: L2-norm-Plancherel}
\|f\|_{L^{2}(\Sp)}^2=\int_{\lambda\in\mathbb R^{\ast}}\|\widehat{f}(\lambda)\|_{{\mathtt{HS}}[L^{2}(\mathbb R^{n})]}^2\, |\lambda|^{n}d\lambda,
\end{equation}
where $\|\cdot\|_{{\mathtt{HS}[L^2(\mathbb R^{n})]}}$ is the Hilbert-Schmidt norm on $L^{2}(\mathbb R^{n})$.

Now, we deal with the linear case of the Cauchy problem \eqref{EQ: NoL-01i}, that is,
\begin{equation}\label{EQ: NoL-02-0}
\left\{ \begin{split}
\partial_{t}^{2}u(t,z)-\H u(t)+b\partial_{t}u(t,z)+m u(t,z)&=0, \\
u(0,z)&=u_{0}(z), \\
\partial_{t}u(0,z)&=u_{1}(z), \\
\hbox{for all} \quad t>0 \quad \hbox{and} \quad &z\in\Sp.
\end{split}
\right.
\end{equation}
By acting by the group Fourier transform on this equation, we obtain
\begin{equation}\label{EQ: NoL-02-FT}
\left\{ \begin{split}
\partial_{t}^{2}\widehat{u}(t, \lambda)+\sigma_\H(\lambda) \widehat{u}(t, \lambda)+b\partial_{t}\widehat{u}(t, \lambda)+m \widehat{u}(t, \lambda)&=0, \quad t>0,\\
\widehat{u}(0, \lambda)&=\widehat{u}_{0}(\lambda), \\
\partial_{t} \widehat{u}(0, \lambda)&=\widehat{u}_{1}(\lambda),
\end{split}
\right.
\end{equation}
where $\sigma_\H(\lambda)$ is the symbol of $-\H$. It takes the form
\begin{equation}\label{EQ: sub-Laplacian-symbol}
\sigma_\H(\lambda)=|\lambda|\mathrm H_w\equiv |\lambda|\sum_{j=1}^{n}(-\partial_{w_{j}}^{2}+w_{j}^{2}),
\end{equation}
where $\mathrm H_{w}:=\sum_{j=1}^{n}(-\partial_{w_{j}}^{2}+w_{j}^{2})$ is the harmonic operator acting on $L^{2}(\mathbb R^{n})$, see e.g. \cite[Section 6.2.1]{FR16}.

Since the harmonic oscillator $\mathrm H_{w}$ is essentially self-adjoint in $L^{2}(\mathbb R^{n})$ and, its system of eigenfunctions $\{\psi_{k}\}_{k=1}^{\infty}$ is a basis in $L^{2}(\mathbb R^{n})$, we have an ordered set of positive numbers $\{\mu_{k}\}_{k=1}^{\infty}$ such that
$$
\mathrm H_{w}\psi_{k}(w)=\mu_{k}\psi_{k}(w), \,\,\, w\in\mathbb R^{n},
$$
for all $k\in\mathbb N$.
More precisely, $\mathrm H_w$ has eigenvalues
$$
\lambda_{k}=\sum_{j=1}^{n}(2k_{j}+1), \,\,\, k=(k_{1}, \ldots, k_{n})\in\mathbb N^{n},
$$
with corresponding eigenfunctions
$$
e_{k}(w)=\prod_{j=1}^{n}P_{k_{j}}(w_{j}){\rm e}^{-\frac{|w|^{2}}{2}},
$$
which form an orthogonal system in $L^{2}(\mathbb R^{n})$. Here, $P_{m}(\cdot)$ is the $m$--th order Hermite polynomial and
$$
P_{m}(t)=c_{m}{\rm e}^{\frac{|t|^{2}}{2}}\left(t-\frac{d}{dt}\right)^{m}{\rm e}^{-\frac{|t|^{2}}{2}},\;
t>0,\; c_{m}=2^{-m/2}(m!)^{-1/2}\pi^{-1/4}.
$$
For more details on these see e.g.  \cite{NiRo:10}.

Consequently, for $(k,l)\in\mathbb N\times\mathbb N$, denoting
\begin{equation}\label{EQ: Harmonic-system of EFs-symbol}
\widehat{u}(t, \lambda)_{kl}:=(\widehat{u}(t, \lambda)\psi_{l}, \psi_{k})_{L^{2}(\mathbb R^{n})},
\end{equation}
we see that the equation \eqref{EQ: NoL-02-FT} is reduced to the system
\begin{equation}\label{EQ: NoL-02-FT-System}
\left\{ \begin{split}
\partial_{t}^{2}\widehat{u}(t, \lambda)_{kl}+b\partial_{t}\widehat{u}(t, \lambda)_{kl}+(|\lambda|\mu_{k}+m)\widehat{u}(t, \lambda)_{kl}&=0, \quad t>0,\\
\widehat{u}(0, \lambda)_{kl}=\widehat{u}_{0}(\lambda)_{kl}\in & L^{2}(\mathbb R^{n}), \\
\partial_{t} \widehat{u}(0, \lambda)_{kl}=\widehat{u}_{1}(\lambda)_{kl}\in & L^{2}(\mathbb R^{n}),
\end{split}
\right.
\end{equation}
for each $\lambda\in\mathbb R^{\ast}$. 

Now, we fix $\lambda\in\mathbb R^{\ast}$ and $(k,l)\in\mathbb N\times\mathbb N$. By solving the second order ordinary differential equation \eqref{EQ: NoL-02-FT-System} with constant coefficients, we get the estimates
\begin{equation}\label{EQ: Est-01-FT-System}
|\widehat{u}(t, \lambda)_{kl}|\lesssim \e[-\frac{b}{2}t]\left[|\widehat{u}_{0}(\lambda)_{kl}| +  (b^{2}/4-|\lambda| \mu_{k}-m)^{-1/2}|\widehat{u}_{1}(\lambda)_{kl}| \right],
\end{equation}
for $2\sqrt{|\lambda| \mu_{k}+m}<b$, and
\begin{equation}\label{EQ: Est-01b-FT-System}
|\widehat{u}(t, \lambda)_{kl}|\lesssim \e[-\frac{b}{2}t]\left[\left(1+\frac{b}{2}t\right)|\widehat{u}_{0}(\lambda)_{kl}| +  t|\widehat{u}_{1}(\lambda)_{kl}| \right],
\end{equation}
for $2\sqrt{|\lambda| \mu_{k}+m}=b$, and
\begin{equation}\label{EQ: Est-02-FT-System}
|\widehat{u}(t, \lambda)_{kl}|\lesssim \e[-(\frac{b}{2}-\sqrt{\frac{b^{2}}{4}-|\lambda| \mu_{k}-m})t]\left[|\widehat{u}_{0}(\lambda)_{kl}| +  (|\lambda| \mu_{k}+m-b^{2}/4)^{-1/2}|\widehat{u}_{1}(\lambda)_{kl}| \right],
\end{equation}
for $b<2\sqrt{|\lambda| \mu_{k}+m}$. Thus, there exists a positive constant $\delta>0$ such that
in all the cases we have
\begin{equation}\label{EQ: Est-03-FT-System-S>1}
\begin{split}
|\,|\lambda| \mu_{k}&+m-b^{2}/4|^{1/2}|\widehat{u}(t, \lambda)_{kl}|\\
&\lesssim \e[-\delta t]\left[|\,|\lambda| \mu_{k}+m-b^{2}/4|^{1/2} |\widehat{u}_{0}(\lambda)_{kl}| +  |\widehat{u}_{1}(\lambda)_{kl}| \right].
\end{split}
\end{equation}

Consequently, we obtain
\begin{equation}\label{EQ: Est-04-FT}
\begin{split}
\|(&1-\sigma_{\H}(\lambda))^{1/2}|\widehat{u}(t, \lambda)\|_{\hil}^{2}=\sum_{k, l=1}^{\infty}(\,1+|\lambda| \mu_{k})\widehat{u}(t, \lambda)_{kl}|^{2}\\
&\lesssim \e[-\delta t]\left[\sum_{k, l=1}^{\infty}(1+|\lambda| \mu_{k})|\widehat{u}_{0}(\lambda)_{kl}|^{2} + \sum_{k, l=1}^{\infty} |\widehat{u}_{1}(\lambda)_{kl}|^{2} \right]\\
&\lesssim\e[-\delta t]\left[\|(1-\sigma_{\H}(\lambda))^{1/2}\widehat{u}_{0}(\lambda)\|_{\hil}^{2}+\|\widehat{u}_{0}(\lambda)\|_{\hil}^{2}\right].
\end{split}
\end{equation}
The same estimates work if we multiply the equation \eqref{EQ: NoL-02-FT-System} by powers of the spectral decomposition of the symbol of the sub-Laplacian.

The  Sobolev spaces $H^s_\H$, $s\in\mathbb R$, associated to
$\H$, are defined as
\begin{equation*}\label{EQ:HsL-00}
H^s_\H(\Sp):=\left\{ f\in\Dcal'(\Sp): (I-\H)^{s/2}f\in
L^2(\Sp)\right\},
\end{equation*}
with the norm $\|f\|_{H^s_\H(\Sp)}:=\|(I-\H)^{s/2}f\|_{L^2(\Sp)}.$
We refer to Folland \cite{F75} for a thorough analysis of these spaces and their properties.


For the solution of the system \eqref{EQ: NoL-02-FT-System}, for each  $\widehat{u}(t, \lambda)_{kl}$ for fixed $(k,l)\in\mathbb N\times\mathbb N$, we obtain an explicit formula
\begin{equation}\label{EQ: Est-01-FT-System-00}
\begin{split}
\widehat{u}(t,\lambda)_{kl}=&[\left(\frac{b}{4i\sqrt{|\lambda| \mu_{k}+m-b^{2}/4}}+\frac{1}{2}\right)\e[(-b/2+i\sqrt{|\lambda| \mu_{k}+m-b^{2}/4})t]\\
&+\left(\frac{i b}{4\sqrt{|\lambda| \mu_{k}+m-b^{2}/4}}+\frac{1}{2}\right)\e[(-b/2-i\sqrt{|\lambda| \mu_{k}+m-b^{2}/4})t]]\widehat{u}_{0}(\lambda)_{kl}\\ &+[\frac{1}{2i\sqrt{|\lambda| \mu_{k}+m-b^{2}/4}}\e[(-b/2+i\sqrt{|\lambda| \mu_{k}+m-b^{2}/4})t]\\
&+\frac{i}{2\sqrt{|\lambda| \mu_{k}+m-b^{2}/4}}\e[(-b/2-i\sqrt{|\lambda| \mu_{k}+m-b^{2}/4})t]]\widehat{u}_{1}(\lambda)_{kl}.
\end{split}
\end{equation}
To obtain similar Sobolev estimates for negative $s$ we consider cases $b<2\sqrt{|\lambda| \mu_{k}+m}$ and $2\sqrt{|\lambda| \mu_{k}+m}<b$, and then the case $2\sqrt{|\lambda| \mu_{k}+m}\approx b$.

When $b<2\sqrt{|\lambda| \mu_{k}+m}$ let us denote $a_{k}:=\sqrt{|\lambda| \mu_{k}+m-b^{2}/4}$. Then by the direct calculations we get
\begin{equation}\label{EQ: Est-01-FT-System-01}
\begin{split}
\widehat{u}(t,\lambda)_{kl}=& \e[-(b/2) t]  \left[\frac{b\sin(a_{k}t)}{2 a_{k}}\right]\widehat{u}_{0}(\lambda)_{kl}+\e[-(b/2) t]  \left[\frac{\sin(a_{k}t)}{a_{k}}\right]\widehat{u}_{1}(\lambda)_{kl}.
\end{split}
\end{equation}

When $b>2\sqrt{|\lambda| \mu_{k}+m}$ we denote $c_{k}:=\sqrt{b^{2}/4-|\lambda| \mu_{k}-m}$. Then we obtain
\begin{equation}\label{EQ: Est-01-FT-System-02}
\begin{split}
\widehat{u}(t,\lambda)_{kl}=& \e[-(b/2) t]  \left[\frac{b\sinh(c_{k}t)}{2 c_{k}}+\cosh(c_{k}t)\right]\widehat{u}_{0}(\lambda)_{kl}+\e[-(b/2) t]  \left[\frac{\sinh(c_{k}t)}{c_{k}}\right]\widehat{u}_{1}(\lambda)_{kl}.
\end{split}
\end{equation}

We observe that
\begin{equation}\label{EQ: Est-01-FT-System-03}
\begin{split}
\frac{\sin(a_{k}t)}{a_{k}}&=t+o(1), \,\,\, \hbox{as} \,\,\, a_{k}\sim0; \\
\frac{b\sinh(c_{k}t)}{2 c_{k}}+\cosh(c_{k}t)&=\frac{bt}{2}+1+o(1), \,\,\, \hbox{as} \,\,\, c_{k}\sim0;\\
\frac{\sinh(c_{k}t)}{c_{k}}&=t+o(1), \,\,\, \hbox{as} \,\,\, c_{k}\sim0.\\
\end{split}
\end{equation}

Let us now define a characteristic function $\chi\in C_{0}^{\infty}([0, \infty))$ as
\begin{equation*}
\chi(t)=\left\{\begin{split}
1, \,\,\, |t-b^{2}/4+m|<1;\\
0, \,\,\, |t-b^{2}/4+m|>2.
\end{split}
\right.
\end{equation*}
Then for any $s\in\mathbb R$ we have
\begin{equation}\label{EQ: s-negative Sobolev understand}
\|w\|_{H_{\H}^{s}(\Sp)}\simeq \|\chi(\H)w\|_{H_{\H}^{s}(\Sp)}+\|(1-\chi(\H))w\|_{H_{\H}^{s}(\Sp)}.
\end{equation}
Since for any $s_{1}, s_{2}\in\mathbb R$ we have 
\begin{equation}\label{EQ: s-negative Sobolev understand-2}
\|\chi(\H)w\|_{H_{\H}^{s_{1}}(\Sp)}\cong\|\chi(\H)w\|_{H_{\H}^{s_{2}}(\Sp)},
\end{equation}
in view of \eqref{EQ: Est-01-FT-System-03},
the estimate \eqref{EQ: Est-01b-FT-System} extends 
to the estimate in Sobolev spaces.
The estimate for $\|(1-\chi(\H))w\|_{H_{\H}^{s}(\Sp)}$ for any $s$ works in the same way as for $s\geq 1$.
Therefore, summarising the arguments above, we obtain

\begin{prop}\label{PR: L2}
Let $s\in\mathbb R$ and assume that $u_{0}\in H^{s}_{\H}(\Sp)$ and $u_{1}\in H^{s-1}_{\H}(\Sp)$. Then there exists a positive constant $\delta>0$ such that
\begin{equation}\label{EQ: Est-05}
\begin{split}
\|u(t, z)\|_{H^{s}_{\H}(\Sp)}^{2}&=\|(1-\sigma_\H(\lambda))^{s/2}\widehat{u}(t, \lambda)\|_{L^{2}(\widehat{\Sp})}^{2}\\
&=\int_{\mathbb R^{\ast}}\|(1-\sigma_\H(\lambda))^{s/2}\widehat{u}(t, \lambda)\|_{\hil}^{2}\, |\lambda|^{n}d\lambda\\
&\lesssim \e[ - 2 \delta t] (\|u_{0}\|_{H^{s}_{\H}(\Sp)}^{2}+\|u_{1}\|_{H^{s-1}_{\H}(\Sp)}^{2})
\end{split}
\end{equation}
holds for all $t>0$.

Moreover, for any $\alpha\in\mathbb N_{0}$ we have
\begin{equation*}\label{EQ: Est-05-second}
\begin{split}
\|\partial_{t}^{\alpha} u(t, z)\|_{H^{s}_{\H}(\Sp)}^{2}&=\|(1-\sigma_\H(\lambda))^{s/2}\partial_{t}^{\alpha}\widehat{u}(t, \lambda)\|_{L^{2}(\widehat{\Sp})}^{2}\\
&=\int_{\mathbb R^{\ast}}\|(1-\sigma_\H(\lambda))^{(\alpha+s)/2}\widehat{u}(t, \lambda)\|_{\hil}^{2}\, |\lambda|^{n}d\lambda\\
&\lesssim \e[ - 2 \delta t] (\|u_{0}\|_{H^{\alpha+s}_{\H}(\Sp)}^{2}+\|u_{1}\|_{H^{\alpha+s-1}_{\H}(\Sp)}^{2})
\end{split}
\end{equation*}
for all $t>0$.
\end{prop}

\section{Semilinear damped wave equations on the Heisenberg group}
\label{SEC:semlinear}

In this section we consider the semilinear wave equation for the sub-Laplacian $\H$ on the Heisenberg group $\Sp$:
\begin{equation}\label{EQ: NoL-01-Semilinear}
\left\{ \begin{split}
\partial_{t}^{2}u(t)-\H u(t)+b\partial_{t}u(t)+m u(t)&=f(u), \quad t>0,\\
u(0)&=u_{0}\in L^{2}(\Sp), \\
\partial_{t}u(0)&=u_{1}\in L^{2}(\Sp).
\end{split}
\right.
\end{equation}
The main case of interest may be of
\begin{equation}\label{EQ:nonlinex}
f(u)=\mu |u|^{p-1} u,
\end{equation}
for $p>1$ and $\mu\in\mathbb C$, but we can treat a more general situation of $f$ satisfying conditions \eqref{PR: f-007i}, see also \eqref{PR: f-007}.

We now recall the Gagliardo--Nirenberg inequality on the Heisenberg group $\Sp$,
see e.g.  Folland \cite{F75} and Varopoulos \cite{V86}, and also \cite{ChR13} for derivation of the best constants there:

\begin{prop}\label{TH: Gagliardo--Nirenberg inequality}
Let $n\geq 1$, $2 \leq q \leq 2 + 2/n$, and let $Q:=2n+2$ be the homogeneous dimension of $\Sp$.
Then for $\theta=\frac{Q(q-2)}{2q}$ the following Gagliardo--Nirenberg inequality is true
\begin{equation}
\label{EQ: GN-inequality}
\|u\|_{L^{q}(\Sp)}\lesssim\|\nabla_{\Sp}u\|_{L^{2}(\Sp)}^{\theta} \|u\|_{L^{2}(\Sp)}^{1-\theta},
\end{equation}
where $\nabla_{\Sp}=(X_1,\ldots,X_n,Y_1,\ldots,Y_n)$ is the horizontal gradient on $\Sp$.
\end{prop}

We now formulate our main result for the Heisenberg group $\Sp$.

\begin{thm} \label{TH: 01}
Let $b>0$ and $m>0$. Assume that $f$ satisfies the properties
\begin{equation} \label{PR: f-007}
\left\{
\begin{split}
f(0) & =0, \\
|f(u)-f(v)| & \leq C (|u|^{p-1}+|v|^{p-1})|u-v|,
\end{split}
\right.
\end{equation}
for some $1 < p \leq 1 + 1/n$.
Assume that the Cauchy data $u_{0}\in H_{\H}^{1}(\Sp)$ and $u_{1}\in L^2(\Sp)$ satisfy
\begin{equation} \label{EQ: Th-cond-01}
\|u_{0}\|_{H_{\H}^{1}(\Sp)}+\|u_{1}\|_{L^2(\Sp)}\leq\varepsilon.
\end{equation}
Then, there exists a small positive constant $\varepsilon_{0}>0$ such that the Cauchy problem \begin{equation*}\label{EQ: NoL-02}
\left\{ \begin{split}
\partial_{t}^{2}u(t)-\H u(t)+b\partial_{t}u(t) + m u(t)&=f(u), \quad t>0,\\
u(0)&=u_{0}\in H_{\H}^{1}(\Sp), \\
\partial_{t}u(0)&=u_{1}\in L^2(\Sp),
\end{split}
\right.
\end{equation*}
has a unique global solution $u\in C(\mathbb R_{+}; H_{\H}^{1}(\Sp))\cap C^{1}(\mathbb R_{+}; L^2(\Sp))$ for all $0<\varepsilon\leq\varepsilon_{0}$.

Moreover, there is a positive number $\delta_{0}>0$ 
such that
\begin{equation}\label{TH-NoL-1-02}
\|\partial_{t}^{\alpha}\H^{\beta}u(t)\|_{L^2(\Sp)}\lesssim \e[-\delta_{0} t],
\end{equation}
for $(\alpha, \beta)=(0, 0)$, or $(\alpha, \beta)=(0, 1/2)$, or $(\alpha, \beta)=(1, 0)$.
\end{thm}

As noted in the introduction, an example of $f$ satisfying \eqref{PR: f-007} is given by \eqref{EQ:nonlinex} or, more generally, by differentiable functions $f$ such that
$$|f'(u)|\leq C|u|^{p-1}.$$

\begin{proof}[Proof of Theorem \ref{TH: 01}]
Let us consider the closed subset $Z$ of the space $C^{1}(\mathbb R_{+}; \,\, H^{1}_{\H}(\Sp))$ defined as
$$
Z :=\{u\in C^{1}(\mathbb R_{+}; \,\, H^{1}_{\H}(\Sp)): \,\, \|u\|_{Z}\leq L\},
$$
with
\begin{align*}
\|u\|_{Z}:=\sup_{t\geq0}\{(1+t)^{-1/2}\e[\delta t](\|u(t, \cdot)\|_{L^2(\Sp)}+\|\partial_{t}u(t, \cdot)\|_{L^2(\Sp)}+\|\H^{1/2}u(t, \cdot)\|_{L^2(\Sp)})\},
\end{align*}
where $L>0$ will be specified later.
Now we define the mapping $\Gamma$ on $Z$ by
\begin{equation}
\label{MAP-01}
\begin{split}
\Gamma[u](t):= u_{{\rm lin}}(t)+\int_{0}^{t}K[f(u)](t-\tau)d\tau,
\end{split}
\end{equation}
where $u_{{\rm lin}}$ is the solution of the linear equation, and $K[f]$ is the solution of the following linear problem:
\begin{equation*}\label{EQ: NoL-02-Kf}
\left\{ \begin{split}
\partial_{t}^{2}w(t)-\H w(t)+b\partial_{t}w(t) + m w(t)&=0, \quad t>0,\\
w(0)&=0, \\
\partial_{t}w(0)&=f.
\end{split}
\right.
\end{equation*}

We claim that
\begin{equation}
\label{MAP-02}
\|\Gamma[u]\|_{Z}\leq L
\end{equation}
for all $u\in Z$ and
\begin{equation}
\label{MAP-03}
\|\Gamma[u]-\Gamma[v]\|_{Z}\leq \frac{1}{r} \|u-v\|_{Z}
\end{equation}
for all $u, v\in Z$ with $r>1$. Once we proved inequalities \eqref{MAP-02} and \eqref{MAP-03}, it follows that $\Gamma$ is a contraction mapping on $Z$. The Banach fixed point theorem then implies that $\Gamma$ has a unique fixed point in $Z$. It means that there exists a unique global solution $u$ of the equation
$$
u=\Gamma[u] \,\,\, \hbox{in} \,\,\, Z,
$$
which also gives the solution to \eqref{EQ: NoL-01-Semilinear}.
So, we now concentrate on proving \eqref{MAP-02} and \eqref{MAP-03}.

Recalling the second assumption in \eqref{PR: f-007} on $f$, namely,
$$
|f(u)-f(v)|\leq C (|u|^{p-1}+|v|^{p-1})|u-v|,
$$
applying it to functions $u=u(t)$ and $v=v(t)$ we get
$$
\|(f(u)-f(v))(t, \cdot)\|_{L^2(\Sp)}^{2}\leq C \int_{\Sp} (|u(t,z)|^{p-1}+|v(t,z)|^{p-1})^{2}|u(t,z)-v(t,z)|^{2}dz.
$$
Consequently, by the H\"{o}lder inequality, we get
$$
\|(f(u)-f(v))(t, \cdot)\|_{L^2(\Sp)}^{2}\leq C (\|u(t, \cdot)\|^{p-1}_{L^{2p}(\Sp)}+\|v(t, \cdot)\|^{p-1}_{L^{2p}(\Sp)})^{2} \|(u-v)(t, \cdot)\|^{2}_{L^{2p}(\Sp)}
$$
since
$
\frac{1}{\frac{p}{p-1}}+\frac{1}{p}=1.
$
By the Gagliardo--Nirenberg inequality \eqref{EQ: GN-inequality}, and by Young's inequality
$$
a^{\theta} b^{1-\theta}\leq \theta a + (1-\theta) b
$$
for $0\leq\theta\leq1$, $a,b\geq0$,
we obtain
\begin{equation}
\label{EQ: Gagliardo-Nirenberg-01}
\begin{split}
\|(f(u)& -f(v))(t, \cdot)\|_{L^{2}(\Sp)} \leq C \Big[\left(\|\H^{1/2} u(t, \cdot)\|_{L^{2}(\Sp)}+\|u(t, \cdot)\|_{L^{2}(\Sp)}\right)^{p-1}\\
& +\left(\|\H^{1/2} v(t, \cdot)\|_{L^{2}(\Sp)}+\|v(t, \cdot)\|_{L^{2}(\Sp)}\right)^{p-1}\Big]
\\
& \times \left(\|\H^{1/2} (u-v)(t, \cdot)\|_{L^{2}(\Sp)}+\|(u-v)(t, \cdot)\|_{L^{2}(\Sp)}\right).
\end{split}
\end{equation}
Recalling that $\|u\|_{Z}\leq L$ and $\|v\|_{Z}\leq L$, from \eqref{EQ: Gagliardo-Nirenberg-01} we get
\begin{equation}
\label{EQ: Gagliardo-Nirenberg-02}
\|(f(u)-f(v))(t, \cdot)\|_{L^{2}(\Sp)}  \leq C (1+t)^{p/2}\e[-\delta p t] L^{p-1} \|u-v\|_{Z}.
\end{equation}

By putting $v=0$ in \eqref{EQ: Gagliardo-Nirenberg-02}, and using that $f(0)=0$, we also have
\begin{equation}
\label{EQ: Gagliardo-Nirenberg-03}
\begin{split}
\|f(u)(t, \cdot)\|_{L^{2}(\Sp)} & \leq C (1+t)^{p/2}\e[-\delta p t] L^{p}.
\end{split}
\end{equation}

Now, let us estimate the integral operator
\begin{equation}
\label{OP: Int-NoL-01}
\begin{split}
J[u](t,z):=\int_{0}^{t}K[f(u(\tau,z))](t-\tau)d\tau.
\end{split}
\end{equation}
More precisely, for $\alpha=0, 1$ and for all $\beta\geq 0$ we have
\begin{equation*}
\label{OP: Int-NoL-02}
\begin{split}
|\partial^{\alpha}_{t} & \H^{\beta}J[u](t, z)|^{2}\leq \Big| \int_{0}^{t}\partial^{\alpha}_{t}\H^{\beta}K[f(u(\tau,z))](t-\tau) d \tau \Big|^{2} \\
&\leq \left(\int_{0}^{t} \Big| \partial^{\alpha}_{t}\H^{\beta}K[f(u(\tau,z))](t-\tau) \Big| d \tau \right)^{2} \\
&\leq t \int_{0}^{t} \Big| \partial^{\alpha}_{t}\H^{\beta}K[f(u(\tau,z))](t-\tau) \Big|^{2} d \tau.
\end{split}
\end{equation*}

Then by using Proposition \ref{PR: L2}, for $(\alpha, \beta)=(0, 0)$, $(\alpha, \beta)=(0, 1/2)$ and $(\alpha, \beta)=(1, 0)$ we get
\begin{equation}
\label{OP: Int-NoL-03}
\begin{split}
&\|\partial^{\alpha}_{t}\H^{\beta} J[u](t, \cdot)\|_{L^{2}(\Sp)}^{2} \leq t \int_{0}^{t}\| \partial^{\alpha}_{t}\H^{\beta} K[f(u(\tau,z))](t-\tau)  \|_{L^{2}(\Sp)}^{2} d \tau \\
&\leq C t \int_{0}^{t} \e[-2\delta(t-\tau)] \|f(u(\tau, \cdot)) \|_{L^{2}(\Sp)}^{2} d \tau \\
& = C t \e[ - 2 \delta t ] \int_{0}^{t} \e[ 2 \delta \tau ] \|f(u(\tau, \cdot)) \|_{L^{2}(\Sp)}^{2} d \tau.
\end{split}
\end{equation}


Thus, using \eqref{EQ: Gagliardo-Nirenberg-02} and \eqref{EQ: Gagliardo-Nirenberg-03},
we obtain from \eqref{OP: Int-NoL-03} that
\begin{equation}
\label{OP: Int-NoL-04}
\|\partial^{\alpha}_{t} \H^{\beta} ( J[u] - J[v] )(t, \cdot)\|_{L^{2}(\Sp)} \leq C t^{1/2} \e[- \delta t] \, L^{p-1}\|u-v\|_{Z},
\end{equation}
and
\begin{equation}
\label{OP: Int-NoL-05}
\|\partial^{\alpha}_{t} \H^{\beta} J[u](t, \cdot)\|_{L^{2}(\Sp)} \leq C t^{1/2} \e[- \delta t] \, L^{p},
\end{equation}
with the estimates \eqref{OP: Int-NoL-04}--\eqref{OP: Int-NoL-05} holding
for $(\alpha, \beta)=(0, 0)$, $(\alpha, \beta)=(0, 1/2)$ and $(\alpha, \beta)=(1, 0)$.

Consequently, by the definition of $\Gamma[u]$ in \eqref{MAP-01} and
using Proposition \ref{PR: L2} for the first term and estimates for $\|J[u]\|_{Z}$ for the second term below, we obtain
\begin{equation}
\label{Gamma: Contraction mapping-01}
\begin{split}
\|\Gamma[u]\|_{Z} & \leq \|u_{{\rm lin}}\|_{Z} + \|J[u]\|_{Z} \\
& \leq C_{1}(\|u_{0}\|_{H^{1}_{\H}(\Sp)}+\|u_{1}\|_{L^{2}(\Sp)}) + C_{2}L^{p},
\end{split}
\end{equation}
for some $C_{1}>0$ and $C_{2}>0$.

Moreover, in the similar way, we can estimate
\begin{equation}
\label{Gamma: Contraction mapping-02}
\|\Gamma[u]-\Gamma[v]\|_{Z} \leq \|J[u] - J[v]\|_{Z} \leq C_{3}L^{p-1} \|u-v\|_{Z},
\end{equation}
for some $C_{3}>0$. Taking some $r>1$, we choose $L:=r C_{1}(\|u_{0}\|_{H^{1}_{\H}(\Sp)}+\|u_{1}\|_{L^{2}(\Sp)})$ with sufficiently small $\|u_{0}\|_{H^{1}_{\H}(\Sp)}+\|u_{1}\|_{L^{2}(\Sp)}<\varepsilon$ so that
\begin{equation}
\label{Gamma: Contraction mapping-03}
C_{2}L^{p}\leq \frac{1}{r} L, \,\,\,\, C_{3}L^{p-1}\leq \frac{1}{r}.
\end{equation}
Then estimates \eqref{Gamma: Contraction mapping-01}--\eqref{Gamma: Contraction mapping-03} imply the desired estimates \eqref{MAP-02} and \eqref{MAP-03}. This means that we can apply the fixed point theorem for the existence of solutions.

The estimate \eqref{TH-NoL-1-02} follows from \eqref{OP: Int-NoL-03}. Theorem \ref{TH: 01} is now proved.
\end{proof}

\section{Nonlinear damped wave equations on graded Lie groups}
\label{SEC:graded}

In this section for a positive Rockland operator $\R$ we will derive the well-posedness results for the semilinear and then for nonlinear wave equation for small Cauchy data. At first, we start by recalling some definitions and notations following Folland and Stein \cite{FS-book} or \cite[Section 3.1]{FR16}. We also establish a new family of Gagliardo-Nirenberg inequalities on graded Lie groups.

\subsection{Gagliardo-Nirenberg inequalities}
\label{SEC:prelim}

A Lie algebra $\mathfrak g$ is called graded when it is endowed with a vector space decomposition
$$\mathfrak g=\bigoplus_{j=1}^\infty V_j\;\textrm{ such that } [V_i,V_j]\subset V_{i+j},$$
and where all but finitely many of $V_j$'s are zero.
Consequently, a connected simply connected Lie group $\G$ is called graded if its Lie algebra $\mathfrak g$ is graded.
A special case of stratified $\G$ arises when the first stratum $V_1$ generates $\mathfrak g$ as an algebra.

Graded Lie groups are necessarily nilpotent. Moreover, they are also homogeneous Lie groups with a canonical choice of dilations. Namely, let us define the operator $A$ by setting
$AX=jX$ for $X\in V_j$. Then the dilations on $\mathfrak g$ are defined by
$$
D_r:={\rm Exp}(A\ln r), \; r>0.
$$
The {\em homogeneous dimension} $Q$ of $\G$ is defined by
$$
Q:=\nu_1+\ldots+\nu_n={\rm Tr}\, A.
$$

From now on let $\G$ be a graded Lie group.
Rockland operators have been originally defined in \cite{Rockland} through the representation theoretic language. Following \cite[Definition 4.1.1]{FR16},
we say that $\R$ is a Rockland operator on
$\G$ if $\R$ is a left-invariant differential operator which is homogeneous of a positive order $\nu\in\mathbb N$
and satisfies the following Rockland condition:
\begin{itemize}
\item
for all representations $\pi\in \widehat{\G}$, excluding the trivial one, the operator $\pi(\R)$ is injective on $\mathcal H^{\infty}_{\pi}$, namely, from
$$
\pi(\R)v=0
$$
it follows that $v=0$ for all $v\in \mathcal H^{\infty}_{\pi}$.
\end{itemize}
Here $\widehat{\G}$ denotes the unitary dual of the graded Lie group $\G$, $\mathcal H^{\infty}_{\pi}$ is the space of smooth vectors of the representation $\pi\in\widehat{\G}$,
and $\pi(\R)$ is the infinitesimal representation (or the symbol) of $\R$ as an element of the universal enveloping algebra of $\G$,
see \cite[Definition 1.7.2]{FR16}.
For more information on graded Lie groups and Rockland operators we refer to \cite[Chapter 4]{FR16}.

It has been shown by Helffer and Nourrigat in \cite{HN-79} that a left-invariant differential operator $\R$ of homogeneous positive degree $\nu\in\mathbb N$ satisfies the Rockland condition if and only if it is hypoelliptic. Such operators are called {Rockland operators}.

So, {\em a left-invariant differential operator is a Rockland operator if and only if it is homogeneous and hypoelliptic}.

The  Sobolev spaces $H^s_\R(\G)$, $s\in\mathbb R$, associated to
positive Rockland operators $\R$ have been analysed in \cite{FR:Sobolev} using heat kernel methods, see also \cite{FR16}. The positivity (of an operator) refers to the positivity in the operator sense. One of the equivalent definitions of Sobolev spaces is
\begin{equation*}\label{EQ:HsL-00-g}
H^s(\G):=H^s_\R(\G):=\left\{ f\in\Dcal'(\G): (I+\R)^{s/\nu}f\in
L^2(\G)\right\},
\end{equation*}
with the norm $\|f\|_{H^s_\R(\G)}:=\|(I+\R)^{s/\nu}f\|_{L^2(\G)},$ for a positive Rockland operator of homogeneous degree $\nu$. Among other things, it has been shown that these Sobolev spaces are independent of a particular choice of the Rockland operator $\R$, so we may omit writing the subscript $\R$.

We now establish a version of the Gagliardo-Nirenberg inequality on graded Lie groups. Some version of such inequality was shown in \cite{BFKG-graded}, and also in
\cite[Theorem 4.4.28, (7)]{FR16}, namely,
for $q,r\in (1,\infty)$ and $0<\sigma<s$ there exists $C>0$ such that
\begin{equation}\label{EQ:GN-old}
\|f\|_{\dot{L}^p_\sigma}\leq C \|f\|_{L^q}^\theta \|f\|_{\dot{L}^r_s}^{1-\theta},
\end{equation}
where $\theta=1-\frac{\sigma}{s}$ and $p\in (1,\infty)$ is given via
$\frac1p=\frac{\theta}{q}+\frac{1-\theta}{r}.$ Here $\dot{L}^p_\sigma$ is the homogeneous Sobolev space defined as the space of all $f\in\Dcal'(\G)$ such that $\R^{\sigma/\nu}f\in L^p(\G)$, where $\R$ is a positive Rockland operator of homogeneous degree $\nu$. Again, these spaces are independent of a particular choice of $\R$, see \cite[Section 4.4]{FR16}.
However, this inequality \eqref{EQ:GN-old} will not be suitable for our purpose, and we establish another version as a consequence of the following Sobolev inequality on $\G$:

\begin{prop}[{\cite[Proposition 4.4.13, (5)]{FR16}}]
\label{CR: Sobolev-Ineq-s on graded Lie Groups}
Let $\G$ be a graded Lie group of homogeneous dimension $Q$. Let $a>0$ and $1<p<q<\infty$ be such that
$$
Q\left(\frac{1}{p}-\frac{1}{q}\right)=a.
$$
Then we have the following Sobolev inequality
\begin{equation}
\label{EQ: Sobolev-inequality on graded Lie Groups}
\|u\|_{L^{q}(\G)}\lesssim \|u\|_{\dot{L}_{a}^{p}(\G)}\simeq
\|\R^{a/\nu}u\|_{L^{p}(\G)},
\end{equation}
for all $u\in \dot{L}^{p}_{a}(\G)$, and  where $\R$ is any positive Rockland operator of homogeneous degree $\nu$.
\end{prop}
If $\G$ is a stratified Lie group, $\R$ is a sub-Laplacian and $\nu=2$, the estimate \eqref{EQ: Sobolev-inequality on graded Lie Groups} was established by Folland \cite{F75}. We refer to \cite[Proposition 4.4.13]{FR16} for other embedding theorems on graded Lie groups.

We now show that the Sobolev inequality implies a family of the Gagliardo--Nirenberg inequalities, one of which is needed for our analysis:

\begin{thm}\label{LM: GN-Ineq-s on graded Lie Groups}
Let $\G$ be a graded Lie group of homogeneous dimension $Q$ and let $\R$ be a positive Rockland operator of homogeneous degree $\nu$.
Assume that
\begin{equation}\label{EQ:GN-ass}
a>0,\quad 1<r<\frac{Q}{a}\; \textrm{ and }\; 1\leq p\leq q\leq \frac{rQ}{Q-ar}.
\end{equation}
Then we have the following Gagliardo-Nirenberg type inequality,
\begin{equation}
\label{EQ: GN-inequality on graded Lie Groups-g}
\|u\|_{L^{q}(\G)}\lesssim \|u\|_{\dot{L}_{a}^{r}(\G)}^{s} \|u\|_{L^{p}(\G)}^{1-s}\simeq
 \|\R^{a/\nu}u\|_{L^{r}(\G)}^{s} \|u\|_{L^{p}(\G)}^{1-s},
\end{equation}
for $s=\left(\frac1p-\frac1q\right) \left(\frac{a}Q+\frac1p-\frac1r\right)^{-1}\in [0,1],$
provided that $\frac{a}Q+\frac1p-\frac1r\not=0$.

If $\frac{a}Q+\frac1p-\frac1r=0$, we have $p=q=\frac{rQ}{Q-ar}$, in which case \eqref{EQ: GN-inequality on graded Lie Groups-g} holds for any $0\leq s\leq 1$.
\end{thm}

\begin{proof} 
By the H\"{o}lder inequality,  we have
$$
\int_{\G}|u|^{q}dx=\int_{\G}|u|^{qs}|u|^{q(1-s)}dx\leq\left(\int_{\G}|u|^{p^{\ast}}dx\right)^{\frac{qs}{p^{\ast}}} \left(\int_{\G}|u|^{p}dx\right)^{\frac{q(1-s)}{p}},
$$
for any $s\in [0,1]$ such that
\begin{equation}\label{EQ:ps1}
\frac{qs}{p^{\ast}}+\frac{q(1-s)}{p}=1.
\end{equation}
Then by using Corollary \ref{CR: Sobolev-Ineq-s on graded Lie Groups}, for $1<r<p^*<\infty$ we obtain
$$
\|u\|_{L^{q}(\G)}\lesssim \|u\|_{\dot{L}_{a}^{r}(\G)}^{s} \|u\|_{L^{p}(\G)}^{1-s},
$$
where
\begin{equation}\label{EQ:ps}
Q\left(\frac{1}{r}-\frac{1}{p^*}\right)=a,
\end{equation}
yielding \eqref{EQ: GN-inequality on graded Lie Groups-g}.
We only have to check that conditions \eqref{EQ:GN-ass} imply that $r<p^*$ and that $s\in [0,1]$.
Indeed, the relation \eqref{EQ:ps} implies that $\frac{1}{p^*}=\frac1r-\frac{a}Q>0$, and so
\eqref{EQ:ps1} gives
$$
s\left(\frac{a}Q+\frac1p-\frac1r\right)=\frac1p-\frac1q\geq 0.
$$
The condition $q\leq \frac{rQ}{Q-ar}$ guarantees that
$\frac{a}Q+\frac1p-\frac1r\geq \frac1p-\frac1q$, so that $s$ is uniquely determined for
$\frac{a}Q+\frac1p-\frac1r\not=0$. We also note that then automatically $s\in [0,1]$.

Assume now that $\frac{a}Q+\frac1p-\frac1r=0$. This implies that
$p=\frac{rQ}{Q-ar}$, so that the conditions \eqref{EQ:GN-ass} imply that
\begin{equation}\label{EQ:ino}
p=q=\frac{rQ}{Q-ar}.
\end{equation}
If $s=0$, \eqref{EQ: GN-inequality on graded Lie Groups-g} trivially holds for $p=q$, so we may assume that $1\geq s>0$. Moreover, we can assume that $\|u\|_{L^q}\not=0$ since otherwise there is noting to prove. Consequently, using that $s>0$, $p=q$ and $\|u\|_{L^q}\not=0$, inequality \eqref{EQ: GN-inequality on graded Lie Groups-g} reduces to the Sobolev inequality in
Proposition \ref{CR: Sobolev-Ineq-s on graded Lie Groups} since we have
$Q(\frac1r-\frac1q)=a$ under conditions \eqref{EQ:ino}, and since
 $r<q$ in view of $\frac1r-\frac1q=\frac{a}Q$ in this case.
\end{proof}

A special case of Theorem \ref{LM: GN-Ineq-s on graded Lie Groups} important for our further analysis is that of $p=r=2$ and $a=1$, in which case we obtain a more classical Gagliardo-Nirenberg inequality:

\begin{cor}\label{CR: GN-Ineq-s on graded Lie Groups}
Let $\G$ be a graded Lie group of homogeneous dimension $Q\geq 3$ and let $\R$ be a positive Rockland operator of homogeneous degree $\nu$. Then for any
$$
2\leq q\leq \frac{2Q}{Q-2} = 2 + \frac{4}{Q-2}
$$
we have the following Gagliardo-Nirenberg type inequality,
\begin{equation}
\label{EQ: GN-inequality on graded Lie Groups}
\|u\|_{L^{q}(\G)}\lesssim
 \|u\|_{\dot{H}^{1}(\G)}^{s} \|u\|_{L^{2}(\G)}^{1-s}\simeq
 \|\R^{1/\nu}u\|_{L^{2}(\G)}^{s} \|u\|_{L^{2}(\G)}^{1-s},
\end{equation}
for $s=s(q)=\frac{Q(q-2)}{2q}\in[0, 1]$.
\end{cor}

We also record another more general special case of Theorem \ref{LM: GN-Ineq-s on graded Lie Groups} with $p=r=2$, but with any $a>0$:

\begin{cor}
Let $\G$ be a graded Lie group of homogeneous dimension $Q$ and let $\R$ be a positive Rockland operator of homogeneous degree $\nu$. Then for any
$$
0<a<\frac{Q}{2}\;\textrm{ and } \;2\leq q\leq \frac{2Q}{Q-2a} = 2 + \frac{4a}{Q-2a}
$$
we have the following Gagliardo-Nirenberg type inequality,
\begin{equation}
\label{EQ: GN-inequality on graded Lie Groups}
\|u\|_{L^{q}(\G)}\lesssim
 \|u\|_{\dot{H}^{a}(\G)}^{s} \|u\|_{L^{2}(\G)}^{1-s}\simeq
 \|\R^{a/\nu}u\|_{L^{2}(\G)}^{s} \|u\|_{L^{2}(\G)}^{1-s},
\end{equation}
for $s=\frac{Q}{a}(\frac12-\frac1q)\in[0, 1]$.
\end{cor}

\subsection{Linear equation}

Now, we are ready to deal with the linear equation
\begin{equation}\label{EQ: NoL-01-g}
\left\{ \begin{split}
\partial_{t}^{2}u(t)+\R u(t)+b\partial_{t}u(t)+m u(t)&=0, \quad t>0,\\
u(0)&=u_{0}\in L^{2}(\G), \\
\partial_{t}u(0)&=u_{1}\in L^{2}(\G),
\end{split}
\right.
\end{equation}
with the damping term determined by $b>0$ and with the mass $m>0$.

Following \cite{FR16}, we briefly recall some definitions related to the Fourier analysis on a graded Lie group $\G$. For $f\in\mathcal S(\G)$ its group Fourier transform is given by
\begin{equation}\label{EQ: Group-FT-g}
\widehat{f}(\pi):=\int_{\G}f(x)\pi(x)^{\ast}dx
\end{equation}
with the representation $\pi\in\widehat{\G}$ realised as the mapping
\begin{equation}\label{EQ: representation-g}
\pi: \mathcal H_{\pi} \to \mathcal H_{\pi},
\end{equation}
where $\mathcal H_{\pi}$ is the representation space of $\pi$, and where we routinely identify $\pi$ with its equivalence class.
The Fourier inversion formula takes the form
\begin{equation}\label{EQ: Group-FT-g}
f(x)=\int_{\widehat{\G}}\mathrm{Tr}[\widehat{f}(\pi)\pi(x)]d\mu(\pi),
\end{equation}
where $\mathrm{Tr}$ is the trace operator, and $d\mu(\pi)$ is the Plancherel measure on $\widehat{\G}$.
The Plancherel theorem says that
\begin{equation}\label{EQ: L2-norm-Plancherel-g}
\|f\|_{L^{2}(\G)}^{2}=\int_{\widehat{\G}}\|\widehat{f}(\pi)\|_{\hil[\mathcal H_{\pi}]}^{2} d\mu(\pi),
\end{equation}
where $\|\cdot\|_{\hil[\mathcal H_{\pi}]}$ is the Hilbert--Schmidt norm on $\mathcal H_{\pi}$.
We refer to \cite{FR16} for details of the Fourier analysis on graded Lie groups.

Now, the group Fourier transform applied to \eqref{EQ: NoL-01-g} gives
\begin{equation}\label{EQ: NoL-02-FT-g}
\left\{ \begin{split}
\partial_{t}^{2}\widehat{u}(t, \pi)+\sigma_\R(\pi) \widehat{u}(t, \pi)+b\partial_{t}\widehat{u}(t, \pi)+m \widehat{u}(t, \pi)&=0, \quad t>0,\\
\widehat{u}(0, \pi)&=\widehat{u}_{0}(\pi), \\
\partial_{t} \widehat{u}(0, \pi)&=\widehat{u}_{1}(\pi),
\end{split}
\right.
\end{equation}
where $\sigma_\R(\pi)=\pi(\R)$ is the symbol of $\R$ given by its infinitesimal representation.
It is known that $\sigma_\R(\lambda)$ has a discrete spectrum in $(0,\infty)$ for any non-trivial representation $\pi\in\widehat{\G}$,
see \cite{HJL}, \cite{tER:97}, and also \cite[Remark 4.2.8, (4)]{FR16}. Therefore, we can decompose \eqref{EQ: NoL-02-FT-g} with respect to the basis of eigenvectors of $\pi(\R)$.
Repeating discussions of Section \ref{SEC:linear}
we obtain
\begin{prop}\label{PR: L2-g}
Let $\R$ be a positive Rockland operator of homogeneous degree $\nu$ on the graded Lie group $\G$. Suppose that $s\in\mathbb R$.
Assume that $u_{0}\in H^{s}(\G)$ and $u_{1}\in H^{s-\nu/2}(\G)$. Then there exists a positive constant $\delta_{1}>0$ such that
\begin{equation}\label{EQ: Est-05-g}
\begin{split}
\|u(t, z)\|_{H^{s}(\G)}^{2}\lesssim \e[ - 2 \delta_{1} t] (\|u_{0}\|_{H^{s}(\G)}^{2}+\|u_{1}\|_{H^{s-\nu/2}(\G)}^{2})
\end{split}
\end{equation}
for all $t>0$.
Moreover, for all $\alpha\in\mathbb N_{0}$ we obtain
\begin{equation*}\label{EQ: Est-05-g-g}
\begin{split}
\|\partial_{t}^{\alpha}u(t, z)\|_{H^{s}(\G)}^{2}\lesssim \e[ - 2 \delta_{1} t] (\|u_{0}\|_{H^{s+\nu\alpha/2}(\G)}^{2}+\|u_{1}\|_{H^{s+(\alpha-1)\nu/2}(\G)}^{2})
\end{split}
\end{equation*}
for any $t>0$.
\end{prop}

\subsection{Semilinear equations}
\label{S-SEC:semlinear--g}

From now on we assume that $\G$ is a graded Lie group of homogeneous dimension $Q\geq 3$.
We now consider the semilinear equation associated to the positive Rockland operator $\R$ of homogeneous degree $\nu$.

\begin{thm} \label{TH: 01-g}
Let $\G$ be a graded Lie group of homogeneous dimension $Q\geq 3$, and let
$\R$ be a positive Rockland operator of homogeneous degree $\nu$.
Let $b>0$ and $m>0$.
Assume that $1 < p \leq 1 + \frac{2}{Q-2}$ and that $f$ satisfies the properties
\begin{equation} \label{PR: f-007-g-g}
\left\{
\begin{split}
f(0) & =0, \\
|f(u)-f(v)| & \leq C (|u|^{p-1}+|v|^{p-1})|u-v|.
\end{split}
\right.
\end{equation}
Assume that the Cauchy data $u_{0}\in H^{\nu/2}(\G)$ and $u_{1}\in L^2(\G)$ satisfy
\begin{equation} \label{EQ: Th-cond-01-g}
\|u_{0}\|_{H^{\nu/2}(\G)}+\|u_{1}\|_{L^2(\G)}\leq\varepsilon.
\end{equation}
Then, there exists a small positive constant $\varepsilon_{0}>0$ such that the Cauchy problem \begin{equation*}\label{EQ: NoL-02-g}
\left\{ \begin{split}
\partial_{t}^{2}u(t)+\R u(t)+b\partial_{t}u(t) + m u(t)&=f(u), \quad t>0,\\
u(0)&=u_{0}\in H^{\nu/2}(\G), \\
\partial_{t}u(0)&=u_{1}\in L^2(\G),
\end{split}
\right.
\end{equation*}
has a unique global solution $u\in C(\mathbb R_{+}; H^{\nu/2}(\G))\cap C^{1}(\mathbb R_{+}; L^2(\G))$ for all $0<\varepsilon\leq\varepsilon_{0}$.

Moreover, there is a positive number $\delta_{2}>0$ 
such that
\begin{equation}\label{TH-NoL-1-02-g}
\|\partial^{\alpha}_{t}\R^{\beta}u(t)\|_{L^2(\G)}\lesssim \e[-\delta_{2} t],
\end{equation}
for all $(\alpha, \beta)\in\mathbb N_{0}\times\frac{1}{\nu}\mathbb N_{0}$ and $\alpha+\nu\beta\leq\nu/2$.
\end{thm}

\begin{proof}
Similarly to the Heisenberg group case, we introduce the closed subsets $Z_{\R}$ of the space $C(\mathbb R_{+}; \,\, H^{1}(\G))$ defined by
$$
Z_{\R} :=\{u\in C^{1}(\mathbb R_{+};  H^{1}(\G)): \,\, \|u\|_{Z_{\R}}\leq L_{\R}\},
$$
with
\begin{align*}
\|u\|_{Z_{\R}}:=\sup_{t\geq0}\left\{(1+t)^{-1/2}\e[\delta_{1} t]\left(\sum\limits_{(\alpha, \beta)\in\mathbb N_{0}\times\frac{1}{\nu}\mathbb N_{0}}^{\alpha+\nu\beta\leq\nu/2}\|\partial^{\alpha}_{t}\R^{\beta}u(t, \cdot)\|_{L^2(\G)}\right)\right\},
\end{align*}
where $L_{\R}>0$ will be specified later. Indeed, the last sum is taken over terms $(\alpha, \beta)=\{(0, 0), (1,0), (0,1/\nu), \ldots, (0,\left[\frac{\nu}{2}\right]1/\nu)\}$.

Now we define the mapping $\Gamma_{\R}$ on $Z_{\R}$ by
\begin{equation}
\label{MAP-01-g}
\begin{split}
\Gamma_{\R}[u](t):= u_{{\rm lin}}(t)+\int_{0}^{t}K_{\R}[f(u)](t-\tau)d\tau,
\end{split}
\end{equation}
where $u_{{\rm lin}}$ is the solution of the linear equation, and $K_{\R}[f]$ is the solution of the following linear problem:
\begin{equation*}\label{EQ: NoL-02-Kf-g}
\left\{ \begin{split}
\partial_{t}^{2}w(t)+\R w(t)+b\partial_{t}w(t) + m w(t)&=0, \quad t>0,\\
w(0)&=0, \\
\partial_{t}w(0)&=f.
\end{split}
\right.
\end{equation*}

Now, we repeat the discussions of the proof of Theorem \ref{TH: 01}, namely, by the H\"{o}lder inequality, we obtain
$$
\|(f(u)-f(v))(t, \cdot)\|_{L^2(\G)}^{2}\leq C (\|u(t, \cdot)\|^{p-1}_{L^{2p}(\G)}+\|v(t, \cdot)\|^{p-1}_{L^{2p}(\G)})^{2} \|(u-v)(t, \cdot)\|^{2}_{L^{2p}(\G)},
$$
where $\frac{1}{\frac{p}{p-1}}+\frac{1}{p}=1.$ Then by taking into account the Gagliardo-Nirenberg inequality \eqref{EQ: GN-inequality on graded Lie Groups} of Corollary \ref{CR: GN-Ineq-s on graded Lie Groups}, and by Young's inequality, we get
\begin{equation}
\label{EQ: Gagliardo-Nirenberg-01-g}
\begin{split}
\|(f(u)& -f(v))(t, \cdot)\|_{L^{2}(\G)} \leq C \Big[\left(\|\R^{1/\nu} u(t, \cdot)\|_{L^{2}(\G)}+\|u(t, \cdot)\|_{L^{2}(\G)}\right)^{p-1}\\
& +\left(\|\R^{1/\nu} v(t, \cdot)\|_{L^{2}(\G)}+\|v(t, \cdot)\|_{L^{2}(\G)}\right)^{p-1}\Big]
\\
& \times \left(\|\R^{1/\nu} (u-v)(t, \cdot)\|_{L^{2}(\G)}+\|(u-v)(t, \cdot)\|_{L^{2}(\G)}\right).
\end{split}
\end{equation}
Recalling that $\|u\|_{Z_{\R}}\leq L_{\R}$ and $\|v\|_{Z_{\R}}\leq L_{\R}$, from \eqref{EQ: Gagliardo-Nirenberg-01-g} we obtain
\begin{equation}
\label{EQ: Gagliardo-Nirenberg-02-g}
\|(f(u)-f(v))(t, \cdot)\|_{L^{2}(\G)}  \leq C (1+t)^{p/2}\e[-\delta_{1} p t] L_{\R}^{p-1} \|u-v\|_{Z_{\R}}.
\end{equation}
Put $v=0$ in \eqref{EQ: Gagliardo-Nirenberg-02-g}. Then by using that $f(0)=0$, we have
\begin{equation}
\label{EQ: Gagliardo-Nirenberg-03-g}
\begin{split}
\|f(u)(t, \cdot)\|_{L^{2}(\G)} & \leq C (1+t)^{p/2}\e[-\delta_{1} p t] L^{p}_{\R}.
\end{split}
\end{equation}
Now, we estimate the following integral operator
\begin{equation}
\label{OP: Int-NoL-01-g}
\begin{split}
J_{\R}[u](t,z):=\int_{0}^{t}K_{\R}[f(u(\tau,z))](t-\tau)d\tau.
\end{split}
\end{equation}
Since
\begin{equation*}
\label{OP: Int-NoL-02-g}
\begin{split}
|\partial^{\alpha}_{t} & \R^{\beta}J_{\R}[u](t, z)|^{2}\leq \Big| \int_{0}^{t}\partial^{\alpha}_{t}\R^{\beta}K_{\R}[f(u(\tau,z))](t-\tau) d \tau \Big|^{2} \\
&\leq \left(\int_{0}^{t} \Big| \partial^{\alpha}_{t}\R^{\beta}K_{\R}[f(u(\tau,z))](t-\tau) \Big| d \tau \right)^{2} \\
&\leq t \int_{0}^{t} \Big| \partial^{\alpha}_{t}\R^{\beta}K_{\R}[f(u(\tau,z))](t-\tau) \Big|^{2} d \tau,
\end{split}
\end{equation*}
by Proposition \ref{PR: L2-g}, for all $(\alpha, \beta)=\{(0, 0), (1,0), (0,1/\nu), \ldots, (0,\left[\frac{\nu}{2}\right]1/\nu)\}$, i.e. for $\alpha+\nu \beta\leq \nu/2$,  we obtain
\begin{equation}
\label{OP: Int-NoL-03-g}
\begin{split}
&\|\partial^{\alpha}_{t}\R^{\beta} J_{\R}[u](t, \cdot)\|_{L^{2}(\G)}^{2} \leq t \int_{0}^{t}\| \partial^{\alpha}_{t}\R^{\beta} K_{\R}[f(u(\tau,z))](t-\tau)  \|_{L^{2}(\G)}^{2} d \tau \\
&\lesssim t \int_{0}^{t} \e[-2\delta_{1}(t-\tau)] \|f(u(\tau, \cdot)) \|_{L^{2}(\G)}^{2} d \tau \\
& = t \e[ - 2 \delta_{1} t ] \int_{0}^{t} \e[ 2 \delta_{1} \tau ] \|f(u(\tau, \cdot)) \|_{L^{2}(\G)}^{2} d \tau.
\end{split}
\end{equation}

Thus, using \eqref{EQ: Gagliardo-Nirenberg-02-g} and \eqref{EQ: Gagliardo-Nirenberg-03-g},
from \eqref{OP: Int-NoL-03-g} we get
\begin{equation}
\label{OP: Int-NoL-04-g}
\|\partial^{\alpha}_{t} \R^{\beta} ( J_{\R}[u] - J_{\R}[v] )(t, \cdot)\|_{L^{2}(\G)} \lesssim t^{1/2} \e[- \delta_{1} t] \, L_{\R}^{p-1}\|u-v\|_{Z_{\R}},
\end{equation}
and
\begin{equation}
\label{OP: Int-NoL-05-g}
\|\partial^{\alpha}_{t} \R^{\beta} J_{\R}[u](t, \cdot)\|_{L^{2}(\G)} \lesssim t^{1/2} \e[- \delta_{1} t] \, L_{\R}^{p},
\end{equation}
with the estimates \eqref{OP: Int-NoL-04-g}--\eqref{OP: Int-NoL-05-g}.

Finally, continuing to discuss as in the above Heisenberg case we obtain the statement of Theorem \ref{TH: 01-g}.
\end{proof}


\subsection{Nonlinear equations}
We note that the techniques of the proof of Theorem \ref{TH: 01-g} allow us to consider the nonlinear equation \eqref{EQ: NoL-02-g} with more general nonlinearities. Namely, instead of $f$ satisfying \eqref{PR: f-007-g-g}  we can deal with the function $F:\mathbb C^{\left[\nu/2\right]}\to\mathbb C$ with the following property:
\begin{equation} \label{RM: F-g}
\left\{
\begin{split}
F(0) & =0, \\
|F(U)-F(V)| & \leq C (|U|^{p-1}+|V|^{p-1})|U-V|,
\end{split}
\right.
\end{equation}
where $U=(\{\R^{j/\nu}u\}_{j=0}^{\left[\frac{\nu}{2}\right]-1})$.
Here $[\frac{\nu}{2}]$ stands for the integer part of $\frac{\nu}{2}.$

\begin{thm} \label{TH: 01-G}
Let $\G$ be a graded Lie group of homogeneous dimension $Q\geq 3$, and let
$\R$ be a positive Rockland operator of homogeneous degree $\nu$.
Let $b>0$ and $m>0$.
Assume that $1 < p \leq 1 + \frac{2}{Q-2}$ and that $F$ satisfies the properties
\eqref{RM: F-g}. Assume that the Cauchy data $u_{0}\in H^{\nu/2}(\G)$ and $u_{1}\in L^2(\G)$ satisfy
\begin{equation} \label{EQ: Th-cond-01-G}
\|u_{0}\|_{H^{\nu/2}(\G)}+\|u_{1}\|_{L^2(\G)}\leq\varepsilon.
\end{equation}
Then, there exists a small positive constant $\varepsilon_{0}>0$ such that the Cauchy problem \begin{equation*}\label{EQ: NoL-02-G}
\left\{ \begin{split}
\partial_{t}^{2}u(t)+\R u(t)+b\partial_{t}u(t) + m u(t)&=F(u, \{\R^{j/\nu}u\}_{j=1}^{\left[\frac{\nu}{2}\right]-1}), \quad t>0,\\
u(0)&=u_{0}\in H^{\nu/2}(\G), \\
\partial_{t}u(0)&=u_{1}\in L^2(\G),
\end{split}
\right.
\end{equation*}
has a unique global solution $u\in C(\mathbb R_{+}; H^{\nu/2}(\G))\cap C^{1}(\mathbb R_{+}; L^2(\G))$ for all $0<\varepsilon\leq\varepsilon_{0}$.

Moreover, there is a positive number $\delta_{3}>0$ 
such that
\begin{equation}\label{TH-NoL-1-02-G}
\|u(t)\|_{L^2(\G)}+\|\R^{1/\nu}u(t)\|_{L^2(\G)}\lesssim \e[-\delta_{3} t].
\end{equation}
\end{thm}
\begin{proof}
As in the proof of Theorem \ref{TH: 01-g}, we consider the closed subset $Z_{1,\R}$:
$$
Z_{1,\R} :=\{u\in C(\mathbb R_{+};  H^{1}(\G)): \,\, \|u\|_{Z_{1,\R}}\leq L_{1,\R}\},
$$
with the norm
\begin{align*}
\|u\|_{Z_{1,\R}}:=\sup_{t\geq0}\{(1+t)^{-1/2}\e[\delta_{1} t]\left(\|u(t, \cdot)\|_{L^2(\G)}+\|\R^{1/\nu}u(t, \cdot)\|_{L^2(\G)}\right)\}.
\end{align*}
Then similarly to the inequality \eqref{OP: Int-NoL-03-g}, we have
\begin{equation}
\label{OP: Int-NoL-03-G}
\begin{split}
&\|\R^{\beta} J_{\R}[u](t, \cdot)\|_{L^{2}(\G)}^{2} \lesssim  t \e[ - 2 \delta_{1} t ] \int_{0}^{t} \e[ 2 \delta_{1} \tau ] \|F(u, \{\R^{j/\nu}u\}_{j=1}^{\left[\frac{\nu}{2}\right]-1}) \|_{H^{\frac{\nu}{2}(2\beta-1)}(\G)}^{2} d \tau.
\end{split}
\end{equation}
Here we need to control
$
\|F(u, \{\R^{j/\nu}u\}_{j=1}^{\left[\frac{\nu}{2}\right]-1}) \|_{H^{\frac{\nu}{2}(2\beta-1)}(\G)}
$
with $\frac{\nu}{2}(2\beta-1)\leq0$. By using the Gagliardo-Nirenberg inequality \eqref{EQ: GN-inequality on graded Lie Groups} of Corollary \ref{CR: GN-Ineq-s on graded Lie Groups}, and by Young's inequality, we obtain
\begin{equation}
\label{EQ: Gagliardo-Nirenberg-01-G}
\begin{split}
\|&F(u, \{\R^{j/\nu}u\}_{j=1}^{\left[\frac{\nu}{2}\right]-1}) - F(v, \{\R^{j/\nu}v\}_{j=1}^{\left[\frac{\nu}{2}\right]-1}) \|_{H^{\frac{\nu}{2}(2\beta-1)}(\G)} \\
&\lesssim \Big[\left(\|\R^{1/\nu} u(t, \cdot)\|_{L^{2}(\G)}+\|u(t, \cdot)\|_{L^{2}(\G)}\right)^{p-1}\\
& +\left(\|\R^{1/\nu} v(t, \cdot)\|_{L^{2}(\G)}+\|v(t, \cdot)\|_{L^{2}(\G)}\right)^{p-1}\Big]
\\
& \times \left(\|\R^{1/\nu} (u-v)(t, \cdot)\|_{L^{2}(\G)}+\|(u-v)(t, \cdot)\|_{L^{2}(\G)}\right)
\end{split}
\end{equation}
for $\beta=0$ and $\beta=1/\nu$.
Consequently, repeating the rest of the proof as in the previous proofs, we obtain the statement of  Theorem \ref{TH: 01-G}.
\end{proof}


\end{document}